\documentclass[11pt]{amsart}

\usepackage[margin=1.0in]{geometry}
\usepackage[dvipsnames]{xcolor}
\usepackage{amsmath,amssymb,amsthm,mathtools}
\usepackage{enumitem}
\usepackage[backref=page]{hyperref}
\usepackage{microtype}

\hypersetup{
    colorlinks=true,
    linkcolor=blue,
    citecolor=Purple,
    urlcolor=blue
}

\renewcommand*{\backref}[1]{}
\renewcommand*{\backrefalt}[4]{%
    \ifcase #1 %
    \or
        \,[#2]%
    \else
        \,[#2]%
    \fi
}

\newtheorem{theorem}{Theorem}[section]
\newtheorem{proposition}[theorem]{Proposition}
\newtheorem{lemma}[theorem]{Lemma}
\newtheorem{corollary}[theorem]{Corollary}
\newtheorem{definition}[theorem]{Definition}
\newtheorem{remark}[theorem]{Remark}
\newtheorem{example}[theorem]{Example}

\DeclareMathOperator{\rank}{rank}

\title{On asymptotic approximate groups in nilpotent groups}
\author{Arindam Biswas}
\email{arin.math@gmail.com}
\date{\today}

\begin{document}

\begin{abstract}
Let \(G\) be a group and let \(A\subseteq G\) be non-empty. We call \(A\) an
asymptotic $(r,l)$-approximate group if, for a fixed dilation factor \(r\), the
larger product sets \(A^{hr}\) can, for all sufficiently large \(h\), be covered
by a bounded number of left translates of \(A^h\), with the bound $l$ independent of
\(h\). We show that, in virtually nilpotent groups, finite sets whose powers contain
a symmetric word ball of radius comparable to \(h\) are asymptotic approximate
groups. We also prove a nonabelian semilinear-set analogue for certain infinite
sets in these groups.
\end{abstract}

\maketitle

\section{Introduction}

Let $G$ be a group.  For subsets $E,F\subseteq G$, write $EF:=\{ef:e\in E,\ f\in F\}$
for their Minkowski set product, and $E^{-1}:=\{e^{-1}:e\in E\}.$
We say that $E$ is \emph{symmetric} if $E^{-1}=E$, and we denote by
\(\langle E\rangle\) the subgroup generated by $E$.  A \emph{subsemigroup} of
$G$ is a subset closed under multiplication; when we need a subsemigroup to
contain the identity, we state this explicitly.

For every subset $E\subseteq G$, we set $E^0:=\{e\}.$
Let $A\subseteq G$ be non-empty.  For an integer $h\ge 1$, write
\[
A^h:=\{a_1a_2\cdots a_h:a_i\in A\}.
\]
Thus $E^h$ always denotes the set of products of exactly $h$ elements of $E$.
If $e\in A$, then the sets $A^h$ are increasing in $h$. Throughout,
\(\mathbb N=\{1,2,3,\dots\}\), and we write
\(\mathbb N_0=\{0,1,2,\dots\}\).

We begin by recalling Nathanson's one-sided notion of approximate group
\cite{Nathanson2018}.  Compared with the finite approximate groups of
Tao and Breuillard--Green--Tao \cite{Tao2008,BGT2012}, this formulation does
not assume symmetry, finiteness, or the presence of the identity.

\begin{definition}[$(r,l)$-approximate group]
Let $r,l\in\mathbb N$.  A non-empty set $A\subseteq G$ is an \emph{$(r,l)$-approximate group} if there is a set $X\subseteq G$ such that
\[
|X|\le l
\quad\text{and}\quad
A^r\subseteq X A.
\]
\end{definition}
He also defined the corresponding asymptotic version.
\begin{definition}[Asymptotic $(r,l)$-approximate group]
Let $r,l\in\mathbb N$.  A non-empty set $A\subseteq G$ is an \emph{asymptotic $(r,l)$-approximate group} if there is $h_0\in\mathbb N$ such that, for every integer $h\ge h_0$, there is a set $X_h\subseteq G$ satisfying
\[
|X_h|\le l
\quad\text{and}\quad
A^{rh}\subseteq X_h A^h.
\]
We say that $A$ is an \emph{asymptotic approximate group} if, for every fixed $r\in\mathbb N$, there exists $l=l(A,r)$ such that $A$ is an asymptotic $(r,l)$-approximate group.
\end{definition}

From a geometric point of view, the inclusion
\[
A^{rh}\subseteq X_hA^h
\]
asserts that $A^{rh}$ can be covered by at most $|X_h|$ left translates of
$A^h$.  The qualifier ``asymptotic'' indicates that this uniform covering
property is required only once the scale parameter $h$ is large enough.
Nathanson proved that every finite subset of an abelian group is an
asymptotic approximate group \cite{Nathanson2018}.  Biswas--Moens gave another
proof, improved bounds in the finite abelian case, and proved corresponding
results for abelian semilinear sets \cite{BiswasMoens2022}.  
In this paper we investigate which subsets of virtually nilpotent groups are
asymptotic approximate groups.  Our main tools are a general inner-ball
criterion and its consequences.  Using them, we obtain an asymmetric
semigroup-generating case, recall the converse established by
Biswas--Moens \cite[Section 5]{BiswasMoens2022}, and prove a restricted
nonabelian semilinear analogue for sets of the form \(FM\), where \(M\) is a
subsemigroup containing \(e\), \(F\subseteq N_G(M)\) is finite, and \(F\) is either
symmetric or semigroup-generates its ambient subgroup.

\subsection{Results obtained}
\begin{theorem}[Symmetric finite sets]
\label{thm:symmetric-main}
Let \(G\) be a virtually nilpotent group and let \(A\subseteq G\) be finite,
non-empty, and symmetric.  Then, for every \(r\in\mathbb N\), there exist
\(l=l(A,r)\in\mathbb N\) and \(h_0\in\mathbb N\) such that, for every
\(h\ge h_0\), there is a set \(X_h\subseteq \langle A\rangle\) with
\[
|X_h|\le l
\quad\text{and}\quad
A^{rh}\subseteq X_hA^h.
\]
Thus \(A\) is an asymptotic approximate group.
\end{theorem}

The second result handles an asymmetric case.  Since $A$ is not assumed symmetric, $A^h$ need not be a word ball.  We impose instead the condition that $A$ generates its subgroup as a semigroup.  This is stronger than group generation but is natural for directed word metrics.

\begin{definition}[Positive semigroup generation]
Let \(A\subseteq G\) be non-empty, and let \(\Gamma:=\langle A\rangle\).  We say
that \(A\) \emph{semigroup-generates} \(\Gamma\) if
\[
\Gamma=\bigcup_{h\ge1}A^h.
\]
Equivalently, every element of \(\Gamma\), including \(e\), is represented by a
non-empty positive word in letters from \(A\).
\end{definition}

\begin{theorem}[Asymmetric semigroup-generating case]
\label{thm:asymmetric-main}
Let \(G\) be a virtually nilpotent group and let \(A\subseteq G\) be finite and
non-empty.  Set \(\Gamma:=\langle A\rangle\).  Assume that \(A\) semigroup-generates
\(\Gamma\), i.e.
\[
\Gamma=\bigcup_{h\ge1}A^h.
\]
Then, for every \(r\in\mathbb N\), there exist \(l=l(A,r)\in\mathbb N\) and
\(h_0\in\mathbb N\) such that, for every \(h\ge h_0\), there is a set
\(X_h\subseteq\Gamma\) with
\[
|X_h|\le l
\quad\text{and}\quad
A^{rh}\subseteq X_hA^h.
\]
Thus \(A\) is an asymptotic approximate group.
\end{theorem}

The semigroup-generation hypothesis is essential for the proof given here,
though it is not necessary in general, as already shown by finite subsets of
abelian groups.  The arbitrary finite asymmetric case in virtually nilpotent
groups is not settled by the present method.  A natural possible approach
would be to prove a directed limit-shape theorem for the product sets
\(A^h\).  See the discussion in Section~\ref{sec:furthercomments}. Finally, going beyond finite sets, we treat a class of normalized infinite sets that plays the role of a nonabelian analogue of semilinear sets.

\begin{theorem}[Sets of the form \(FM\)]
\label{thm:FM-normalized}
Let \(G\) be virtually nilpotent.  Let \(M\subseteq G\) be a subsemigroup
containing \(e\), and let \(F\subseteq N_G(M)\) be finite and non-empty.  Set
\[
A:=FM.
\]
Assume that either
\begin{enumerate}[label=\textup{(\roman*)}]
\item \(F\) is symmetric, or
\item \(F\) semigroup-generates \(\langle F\rangle\).
\end{enumerate}
Then \(A\) is an asymptotic approximate group.
\end{theorem}

\section{Preliminaries}

\subsection{Virtually nilpotent groups and word balls}

A group $G$ is \emph{virtually nilpotent} if it contains a nilpotent subgroup of finite index.  If $G$ is virtually nilpotent and $H\le G$ is finitely generated, then $H$ is virtually nilpotent.  Indeed, if $N\le G$ is nilpotent of finite index, then $H\cap N\le H$ is nilpotent and has finite index in $H$.

Let $S\subseteq \Gamma$ be finite, symmetric, contain $e$, and generate the finitely generated group $\Gamma$.  The word length and word balls are
\[
\ell_S(g):=\min\{n\ge 0:g\in S^n\},
\quad
B_S(n):=\{g\in\Gamma:\ell_S(g)\le n\}.
\]
Since $e\in S$, one has
\[
B_S(n)=S^n
\quad\text{for every integer }n\ge 0.
\]
The associated left-invariant word metric is
\[
d_S(g,h):=\ell_S(g^{-1}h).
\]

\subsection{Bass--Guivarc'h growth estimates}

We shall use the polynomial-growth bounds furnished by the Bass--Guivarc'h
theorem \cite{Bass1972,Guivarch1973} in the following standard form.

\begin{theorem}[Bass--Guivarc'h polynomial growth estimates]\label{thm:bass-guivarch}
Let $\Gamma$ be a finitely generated virtually nilpotent group and let $S$ be a finite symmetric generating set containing $e$.  Then there exist constants $C_1,C_2>0$ and an integer $d=d(\Gamma)\ge 0$ such that, for all integers $n\ge 1$,
\[
C_1 n^d\le |B_S(n)|\le C_2 n^d,
\]
with the convention that if $\Gamma$ is finite one may take $d=0$ after increasing constants. The integer $d$ is the \emph{growth degree} of $\Gamma$, also called the
Bass--Guivarc'h dimension. It is independent of the chosen finite symmetric generating set
$S$. If $\Gamma$ is torsion-free nilpotent with lower central series
\[
\Gamma=\Gamma_1\rhd \Gamma_2\rhd\cdots\rhd \Gamma_{s+1}=\{e\},
\quad
\Gamma_{i+1}=[\Gamma,\Gamma_i],
\]
then
\[
d=\sum_{i=1}^s i\,\rank_{\mathbb Z}(\Gamma_i/\Gamma_{i+1}).
\]
where \(\rank_{\mathbb Z}\) denotes the rank of the finitely generated abelian
group, equivalently
\[
\rank_{\mathbb Z}(\Gamma_i/\Gamma_{i+1})
=
\dim_{\mathbb Q}\bigl((\Gamma_i/\Gamma_{i+1})\otimes_{\mathbb Z}\mathbb Q\bigr).
\]
For general finitely generated virtually nilpotent $\Gamma$, the same integer
is computed after passing to a finite-index torsion-free nilpotent subgroup; the
resulting degree is independent of the chosen subgroup.
\end{theorem}

\begin{remark}
For the arguments below, we use two-sided polynomial bounds, not exact asymptotics.  In particular, the proofs do not depend on the finer Pansu \cite{Pansu1983}, or Breuillard limit-shape results \cite{Breuillard2014}.
\end{remark}

\section{A packing-covering result for word balls}

We also use the following standard covering lemma of Ruzsa.

\begin{lemma}[Ruzsa covering lemma {\cite[Lemma~3.6]{Tao2008}}, \cite{Ruzsa1989}]
\label{lem:ruzsa-covering}
Let \(E,F\) be finite non-empty subsets of a group \(G\).  Suppose that
\[
|EF|\le K|F|
\]
for some \(K\ge 1\).  Then there is a set \(X\subseteq E\) with
\[
|X|\le K
\]
such that
\[
E\subseteq XFF^{-1}.
\]
\end{lemma}

\begin{proof}
See \cite[Lemma~3.6]{Tao2008}, \cite{Ruzsa1989}.
\end{proof}

\begin{proposition}[Polynomial growth covering]\label{prop:covering}
Let $\Gamma$ be a finitely generated virtually nilpotent group and let $S$ be a finite symmetric generating set with $e\in S$.  Fix $R_0\ge 1$ and $0<\theta\le 1$.  Then there exist
\[
L=L(\Gamma,S,R_0,\theta)\in\mathbb N
\quad\text{and}\quad
h_0=h_0(\Gamma,S,R_0,\theta)\in\mathbb N
\]
such that, for every integer $h\ge h_0$, there is a set $Y_h\subseteq \Gamma$ with
\[
|Y_h|\le L
\]
and
\[
B_S(\lfloor R_0 h\rfloor)
\subseteq
Y_h B_S(\lfloor \theta h\rfloor).
\]
Equivalently,
\[
S^{\lfloor R_0h\rfloor}
\subseteq
Y_h S^{\lfloor \theta h\rfloor}.
\]
\end{proposition}

\begin{proof}[Proof of Proposition~\ref{prop:covering}]
If \(\Gamma\) is finite, take \(Y_h=\Gamma\) for all \(h\).  Thus assume that
\(\Gamma\) is infinite.

By Theorem~\ref{thm:bass-guivarch}, there are constants \(C_1,C_2>0\) and an
integer \(d\ge1\) such that
\[
C_1n^d\le |B_S(n)|\le C_2n^d
\quad\text{for all }n\ge1.
\]

For \(h\) large, put
\[
R_h:=\lfloor R_0h\rfloor,\qquad
M_h:=\lfloor \theta h\rfloor,\qquad
q_h:=\left\lfloor\frac{M_h}{2}\right\rfloor .
\]
We shall apply Lemma~\ref{lem:ruzsa-covering} with
\[
E:=B_S(R_h),\qquad F:=B_S(q_h).
\]
Since \(S\) is symmetric and contains \(e\),
\[
EF=B_S(R_h)B_S(q_h)\subseteq B_S(R_h+q_h).
\]
Therefore
\[
\frac{|EF|}{|F|}
\le
\frac{|B_S(R_h+q_h)|}{|B_S(q_h)|}
\le
\frac{C_2}{C_1}\left(\frac{R_h+q_h}{q_h}\right)^d .
\]
For all sufficiently large \(h\),
\[
q_h\ge \frac{\theta h}{3}
\quad\text{and}\quad
R_h+q_h\le (R_0+\theta)h.
\]
Hence
\[
\frac{|EF|}{|F|}
\le
\frac{C_2}{C_1}
\left(\frac{3(R_0+\theta)}{\theta}\right)^d
:=K_0.
\]
By Lemma~\ref{lem:ruzsa-covering}, there is a set
\[
Y_h\subseteq B_S(R_h)
\]
such that
\[
|Y_h|\le K_0
\]
and
\[
B_S(R_h)\subseteq Y_h B_S(q_h)B_S(q_h)^{-1}.
\]
Since \(S\) is symmetric,
\[
B_S(q_h)B_S(q_h)^{-1}
=
B_S(q_h)^2
\subseteq B_S(2q_h)
\subseteq B_S(M_h).
\]
Thus
\[
B_S(\lfloor R_0h\rfloor)
\subseteq
Y_hB_S(\lfloor\theta h\rfloor).
\]
Taking
\[
L:=\left\lceil K_0\right\rceil
\]
and increasing \(h_0\) if necessary proves the proposition.
\end{proof}

\subsection{Alternative derivation}

Proposition~\ref{prop:covering} can also be derived from the large-scale doubling
theorem of Breuillard--Tointon \cite{BreuillardTointon2016}. They prove the following result: for every \(K\ge1\), there exist
constants \(n_0(K)\in\mathbb N\) and \(\Theta(K)\ge1\) such that, if \(S\) is a
finite symmetric subset of a group, \(e\in S\), and
\[
|S^{2n+1}|\le K|S^n|
\]
for some \(n\ge n_0(K)\), then
\[
|S^{cm}|\le \Theta(K)^c |S^m|
\]
for every \(m\ge n\) and every \(c\in\mathbb N\)
\cite[Theorem~1.1]{BreuillardTointon2016}.  The Ruzsa-covering step below is the
same standard maximal-disjoint-translates argument also used in
\cite[Lemma~2.2]{BreuillardTointon2016}.

\begin{proposition}
\label{prop:BT-covering}
Let \(\Gamma\) be a finitely generated virtually nilpotent group, and let
\(S\subseteq\Gamma\) be finite, symmetric, contain \(e\), and generate \(\Gamma\).
Fix \(R_0\ge1\) and \(0<\theta\le1\).  Then there exist
\(L\in\mathbb N\) and \(h_0\in\mathbb N\) such that, for every \(h\ge h_0\),
there is \(Y_h\subseteq\Gamma\) with
\[
|Y_h|\le L
\quad\text{and}\quad
S^{\lfloor R_0h\rfloor}\subseteq Y_hS^{\lfloor\theta h\rfloor}.
\]
\end{proposition}

\begin{proof}
If \(\Gamma\) is finite, take \(Y_h=\Gamma\).  Assume \(\Gamma\) is infinite.

By Bass--Guivarc'h growth estimates, there are constants \(C_1,C_2>0\) and
\(d\ge1\) such that
\[
C_1n^d\le |S^n|\le C_2n^d
\quad\text{for all }n\ge1.
\]
Thus there is \(K\ge1\), depending on \((\Gamma,S)\), such that
\[
|S^{2n+1}|\le K|S^n|
\]
for all sufficiently large \(n\).  Choose \(n_\ast\) large enough that the
hypothesis of \cite[Theorem~1.1]{BreuillardTointon2016} holds at some scale
\(n\ge n_0(K)\) with \(n\le n_\ast\).  Then there is a constant
\(\Theta=\Theta(K)\) such that
\[
|S^{cm}|\le \Theta^c |S^m|
\]
for every \(m\ge n_\ast\) and every \(c\in\mathbb N\).

For \(h\) large, put
\[
R_h:=\lfloor R_0h\rfloor,\qquad
M_h:=\lfloor\theta h\rfloor,\qquad
q_h:=\left\lfloor\frac{M_h}{2}\right\rfloor .
\]
Choose
\[
c:=\left\lceil \frac{3(R_0+\theta)}{\theta}\right\rceil .
\]
After increasing \(h_0\), we may assume that \(q_h\ge n_\ast\) and
\[
R_h+q_h\le c q_h
\]
for all \(h\ge h_0\).

Apply Lemma~\ref{lem:ruzsa-covering} with
\[
E:=S^{R_h},
\qquad
F:=S^{q_h}.
\]
Since
\[
EF\subseteq S^{R_h+q_h}\subseteq S^{cq_h},
\]
the Breuillard--Tointon estimate gives
\[
|EF|\le |S^{cq_h}|\le \Theta^c |S^{q_h}|=\Theta^c |F|.
\]
Therefore Lemma~\ref{lem:ruzsa-covering} gives a set
\(Y_h\subseteq S^{R_h}\) with
\[
|Y_h|\le \Theta^c
\]
and
\[
S^{R_h}\subseteq Y_hS^{q_h}(S^{q_h})^{-1}.
\]
Since \(S\) is symmetric,
\[
S^{q_h}(S^{q_h})^{-1}
=
S^{q_h}S^{q_h}
=
S^{2q_h}
\subseteq S^{M_h}.
\]
Hence
\[
S^{\lfloor R_0h\rfloor}
\subseteq
Y_hS^{\lfloor\theta h\rfloor}.
\]
Taking
\[
L:=\left\lceil \Theta^c\right\rceil
\]
finishes the proof.
\end{proof}

\begin{remark}
The rest of the paper uses
the direct Bass--Guivarc'h proof in order to keep the argument self-contained in
the virtually nilpotent setting.
\end{remark}

\begin{lemma}[Adding the identity at bounded covering cost]
\label{lem:add-identity-padding}
Let \(A\subseteq G\) be non-empty, and suppose that
\[
e\in A^p
\]
for some integer \(p\ge1\).  Put
\[
A_e:=A\cup\{e\}.
\]
Then there exists a finite set \(E\subseteq \langle A\rangle\), depending only on
\(A\) and \(p\), such that
\[
A_e^h\subseteq E A^h
\]
for every \(h\ge1\).

Consequently, if \(A_e\) is an asymptotic \((r,l)\)-approximate group, then \(A\)
is an asymptotic \((r,l|E|)\)-approximate group.
\end{lemma}

\begin{proof}
For each residue \(j\in\{0,1,\dots,p-1\}\), choose an element \(w_j\in A^j\),
with the convention \(w_0=e\).  This is possible because \(A\) is non-empty.
Set
\[
E:=\{w_j^{-1}:0\le j<p\}.
\]

Let \(h\ge1\).  Since \(A_e=A\cup\{e\}\), every element of \(A_e^h\) lies in
\(A^k\) for some \(0\le k\le h\), where \(A^0:=\{e\}\).  Let \(g\in A^k\), and put
\[
n:=h-k.
\]
Choose \(j\in\{0,\dots,p-1\}\) with
\[
j\equiv n \pmod p.
\]
Then \(j\le n\), and \(n-j\) is a non-negative multiple of \(p\).  Since
\(e\in A^p\), we have \(A^m\subseteq A^{m+p}\) for every \(m\ge0\).  Iterating,
\[
w_j g\in A^{j+k}\subseteq A^{j+k+(n-j)}=A^h.
\]
Therefore
\[
g\in w_j^{-1}A^h\subseteq EA^h.
\]
Thus \(A_e^h\subseteq EA^h\) for every \(h\ge1\).

Now suppose \(A_e\) is asymptotic \((r,l)\)-approximate.  For all sufficiently large
\(h\), choose \(X_h\) with \(|X_h|\le l\) and
\[
A_e^{rh}\subseteq X_h A_e^h.
\]
Since \(A^{rh}\subseteq A_e^{rh}\), the first part gives
\[
A^{rh}\subseteq A_e^{rh}\subseteq X_h A_e^h\subseteq X_h E A^h.
\]
Since \(|X_hE|\le l|E|\), the conclusion follows.
\end{proof}

\begin{remark}[Dependence of constants]
The covering number obtained in Proposition~\ref{prop:covering} depends on the growth constants for $(\Gamma,S)$, on $R_0$, and on $\theta$.  It is uniform in $h$.  No attempt is made here to optimize the constant. 
\end{remark}

\section{An inner symmetric ball criterion}

\begin{proposition}\label{prop:inner-ball-criterion}
Let \(G\) be virtually nilpotent, let \(A\subseteq G\) be finite and non-empty, and set
\[
\Gamma:=\langle A\rangle.
\]
Let \(S\subseteq \Gamma\) be a finite symmetric generating set with \(e\in S\).  Suppose
that there exist \(\theta>0\) and \(h_1\in\mathbb N\) such that
\[
B_S(\lfloor \theta h\rfloor)\subseteq A^h
\]
for every \(h\ge h_1\).  Then \(A\) is an asymptotic approximate group.
\end{proposition}

\begin{proof}
Since $G$ is virtually nilpotent and $\Gamma=\langle A\rangle$ is finitely
generated, the subgroup $\Gamma$ is itself virtually nilpotent.  All word balls
below are taken inside $\Gamma$.

Since \(e\in B_S(\lfloor \theta h_1\rfloor)\subseteq A^{h_1}\), we may apply
Lemma~\ref{lem:add-identity-padding} with \(p:=h_1\).  Put
\[
B:=A\cup\{e\}.
\]

If \(\Gamma\) is finite, then for every \(h\ge h_1\) the inner-ball hypothesis gives
\(e\in A^h\).  Taking \(X_h=\Gamma\), we get
\[
A^{rh}\subseteq \Gamma=\Gamma A^h=X_hA^h
\]
for all sufficiently large \(h\).  Hence the result is immediate.

We may therefore assume that \(\Gamma\) is infinite.

Fix \(r\in\mathbb N\), and put
\[
\theta_0:=\min(\theta,1).
\]
Then
\[
B_S(\lfloor \theta_0 h\rfloor)\subseteq B_S(\lfloor \theta h\rfloor)\subseteq A^h\subseteq B^h
\]
for every \(h\ge h_1\).

Since \(B\) is finite and \(S\) generates \(\Gamma\), there is a constant
\[
C:=\max\bigl(1,\max_{b\in B}\ell_S(b)\bigr)<\infty.
\]
Hence
\[
B^{rh}\subseteq B_S(Crh)
\]
for every \(h\ge1\).

Apply Proposition~\ref{prop:covering} to \((\Gamma,S)\) with
\[
R_0=Cr
\quad\text{and}\quad
\theta=\theta_0.
\]
There exist \(L\) and \(h_0\) such that, for every \(h\ge h_0\), there is
\(X_h\subseteq\Gamma\) with \(|X_h|\le L\) and
\[
B_S(Crh)\subseteq X_hB_S(\lfloor\theta_0 h\rfloor).
\]
After increasing \(h_0\) so that \(h_0\ge h_1\), we obtain
\[
B^{rh}
\subseteq B_S(Crh)
\subseteq X_hB_S(\lfloor\theta_0 h\rfloor)
\subseteq X_hB^h.
\]
Thus \(B\) is asymptotic \((r,L)\)-approximate.  Since \(r\) was arbitrary, \(B\) is an
asymptotic approximate group.  Lemma~\ref{lem:add-identity-padding} now implies
that \(A\) is an asymptotic approximate group.
\end{proof}

\section{Symmetric case}

We now prove Theorem~\ref{thm:symmetric-main}.

\begin{proof}
First assume that \(e\in A\).  Since \(A\) is finite, symmetric, contains \(e\), and
generates \(\Gamma:=\langle A\rangle\), one has
\[
B_A(h)=A^h
\]
for every \(h\ge0\).  Thus Proposition~\ref{prop:inner-ball-criterion} applies
with \(S=A\), \(\theta=1\), and \(h_1=1\), showing that \(A\) is an asymptotic
approximate group.

Now let \(A\) be arbitrary finite, non-empty, and symmetric.  Then \(e\in A^2\):
indeed, if \(a\in A\), then \(a^{-1}\in A\), so \(aa^{-1}=e\in A^2\).  Put
\[
A_e:=A\cup\{e\}.
\]
The set \(A_e\) is finite, symmetric, and contains \(e\), so the first part shows
that \(A_e\) is an asymptotic approximate group.  Since \(e\in A^2\),
Lemma~\ref{lem:add-identity-padding} applies with \(p=2\), and implies that
\(A\) is an asymptotic approximate group.
\end{proof}

\begin{corollary}[Uniform covering of symmetric powers]\label{cor:symmetric-covering}
Let \(G\) be virtually nilpotent and let \(A\subseteq G\) be finite, symmetric, and
non-empty.  For every \(r\in\mathbb N\) there is \(L=L(A,r)\) such that
\[
\operatorname{cov}(A^{rh},A^h)\le L
\]
for all sufficiently large \(h\), where
\[
\operatorname{cov}(E,F):=\min\{|X|:E\subseteq XF\}.
\]
\end{corollary}

\begin{proof}
This is exactly the conclusion of Theorem~\ref{thm:symmetric-main}.
\end{proof}

\subsection{A known converse}

\begin{proposition}[Biswas--Moens]
\label{prop:BM-converse}
Let \(G\) be a group and let \(A\subseteq G\) be finite and non-empty.  Suppose that
\(A\) is an asymptotic \((r,\ell)\)-approximate group for some \(r\ge 2\) and
\(\ell\in\mathbb N\).  Then the subgroup \(\langle A\rangle\) is virtually nilpotent.
In particular, if \(A\) generates \(G\), then \(G\) is virtually nilpotent.
\end{proposition}

\begin{proof}
We use the result proved in \cite[Section 5]{BiswasMoens2022}.
if \(A\) is finite and asymptotic \((r,\ell)\)-approximate for some \(r\ge2\),
then the subgroup \(\langle A\rangle\) is virtually nilpotent.  More precisely,
Biswas--Moens first show that the one-sided growth function
\(f(n)=|A^n|\) is polynomially bounded, and then apply the Tao/Breuillard--Green--Tao
inverse theorem for polynomial growth to the finite set \(A\).
\end{proof}

\begin{corollary}[Characterization using finite symmetric generating sets]\label{cor:characterization-symmetric}
For a finitely generated group $G$, the following are equivalent.
\begin{enumerate}[label=\textup{(\roman*)}]
\item $G$ is virtually nilpotent.
\item For every finite non-empty symmetric set $A\subseteq G$, the set $A$ is an asymptotic approximate group.
\item Some finite symmetric generating set $S\subseteq G$ with $e\in S$ is an asymptotic $(r,\ell)$-approximate group for some $r\ge 2$ and some $\ell\in\mathbb N$.
\end{enumerate}
\end{corollary}

\begin{proof}
$(i)\Rightarrow(ii)$ is Theorem~\ref{thm:symmetric-main}.  The implication $(ii)\Rightarrow(iii)$ is immediate because $G$ is finitely generated.  The implication $(iii)\Rightarrow(i)$ is Proposition~\ref{prop:BM-converse}.
\end{proof}

\section{Asymmetric case}

We now prove Theorem~\ref{thm:asymmetric-main}.  

\subsection{Bounded positive length of inverses}

Let $A\subseteq G$ be finite with $e\in A$, and suppose $A$ semigroup-generates $\Gamma:=\langle A\rangle$.  Define the positive $A$-length
\[
\ell_A^+(g):=\min\{n\ge 1:g\in A^n\},
\quad g\in\Gamma.
\]
We count only non-empty positive words.  Thus even $e$ has positive $A$-length,
provided $e\in A^n$ for some $n\ge1$.
This is finite for every $g\in\Gamma$ by semigroup generation.  It is generally asymmetric: $\ell_A^+(g)$ and $\ell_A^+(g^{-1})$ need not be equal.

\begin{lemma}[Uniform positive words for inverses]\label{lem:inverse-bound}
Let $A\subseteq G$ be finite with $e\in A$, and assume that $A$ semigroup-generates $\Gamma:=\langle A\rangle$.  Set
\[
S:=A\cup A^{-1}.
\]
Then there is an integer $q\ge 1$ such that
\[
S\subseteq A^q.
\]
Consequently,
\[
S^n\subseteq A^{qn}
\quad\text{for every }n\ge 0.
\]
\end{lemma}

\begin{proof}
Since $A$ semigroup-generates $\Gamma$, each $a^{-1}$ with $a\in A$ lies in $A^m$ for some $m\ge 1$.  Since $A$ is finite, the number
\[
q:=\max\bigl(1,\max_{a\in A}\ell_A^+(a^{-1})\bigr)
\]
is finite.  If $s\in A^{-1}$, then $s\in A^q$ by definition of $q$ and the fact that $e\in A$, which allows padding shorter words to length exactly $q$.  If $s\in A$, then $s\in A\subseteq A^q$ because $e\in A$ and $q\ge 1$.  Thus $S\subseteq A^q$.  Multiplying $n$ times gives $S^n\subseteq A^{qn}$.
\end{proof}

\begin{lemma}[Comparison of asymmetric powers with symmetric balls]\label{lem:power-comparison}
Under the hypotheses of Lemma~\ref{lem:inverse-bound}, for every $h\ge 0$ one has
\[
A^h\subseteq S^h
\]
and, for every $m\ge 0$,
\[
S^m\subseteq A^{qm}.
\]
In particular,
\[
S^{\lfloor h/q\rfloor}\subseteq A^h
\quad\text{for every }h\ge 0.
\]
\end{lemma}

\begin{proof}
The inclusion $A^h\subseteq S^h$ is immediate from $A\subseteq S$.  The inclusion $S^m\subseteq A^{qm}$ is Lemma~\ref{lem:inverse-bound}.  Finally,
\[
q\lfloor h/q\rfloor\le h,
\]
and since $e\in A$, the powers $A^n$ are increasing in $n$.  Therefore
\[
S^{\lfloor h/q\rfloor}\subseteq A^{q\lfloor h/q\rfloor}\subseteq A^h.
\]
\end{proof}

\subsection{Proof of the asymmetric theorem}

\begin{proof}
First assume that \(e\in A\).  Put
\[
S:=A\cup A^{-1}.
\]
Then \(S\) is finite, symmetric, contains \(e\), and generates
\(\Gamma:=\langle A\rangle\).  By Lemma~\ref{lem:power-comparison}, there exists
\(q\ge1\) such that
\[
B_S(\lfloor h/q\rfloor)=S^{\lfloor h/q\rfloor}\subseteq A^h
\]
for every \(h\ge0\).  Proposition~\ref{prop:inner-ball-criterion} therefore
applies with \(\theta=1/q\) and \(h_1=1\), showing that \(A\) is an asymptotic
approximate group.

Now drop the assumption \(e\in A\).  Since \(A\) semigroup-generates \(\Gamma\),
we have \(e\in A^p\) for some \(p\ge1\).  Put
\[
A_e:=A\cup\{e\}.
\]
Then \(A_e\) is finite, contains \(e\), and semigroup-generates \(\Gamma\).  By the
case already proved, \(A_e\) is an asymptotic approximate group.  Since \(e\in A^p\),
Lemma~\ref{lem:add-identity-padding} implies that \(A\) is an asymptotic approximate
group.
\end{proof}

\begin{corollary}[Characterization using semigroup-generating sets]
\label{cor:characterization-semigroup}
For a finitely generated group \(G\), the following are equivalent.
\begin{enumerate}[label=\textup{(\roman*)}]
\item \(G\) is virtually nilpotent.
\item Every finite non-empty set \(A\subseteq G\) that semigroup-generates
\(\langle A\rangle\) is an asymptotic approximate group.
\item There exists a finite non-empty set \(A\subseteq G\) that semigroup-generates
\(G\) and is asymptotic \((r,\ell)\)-approximate for some \(r\ge2\) and some
\(\ell\in\mathbb N\).
\end{enumerate}
\end{corollary}

\begin{proof}
The implication \((i)\Rightarrow(ii)\) is Theorem~\ref{thm:asymmetric-main}, applied
inside the virtually nilpotent subgroup \(\langle A\rangle\).

The implication \((ii)\Rightarrow(iii)\) follows by choosing any finite symmetric
generating set \(S\) for \(G\).  Such an \(S\) semigroup-generates \(G\).

Finally, \((iii)\Rightarrow(i)\) is Proposition~\ref{prop:BM-converse}.
\end{proof}

\begin{remark}
The only use of semigroup generation is in Lemma~\ref{lem:inverse-bound}, applied after adjoining the identity.  It ensures that each inverse of an element of $B:=A\cup\{e\}$ can be written as a positive word of uniformly bounded length in $B$.  This gives a linear-scale inclusion
\[
S^{\lfloor h/q\rfloor}\subseteq B^h,
\quad S=B\cup B^{-1},
\]
which allows us to replace the asymmetric set $A^h$ by an inner symmetric word ball after a bounded padding step.  Without such an inner symmetric ball, the packing argument for word balls does not directly apply.
\end{remark}

\begin{example}[A semigroup-generating asymmetric set]
Let $\Gamma=\mathbb Z$ and let
\[
A=\{-1,0,1,2\}.
\]
Then \(A\) is not symmetric, since \(2\in A\) but \(-2\notin A\).
Nevertheless $A$ semigroup-generates $\mathbb Z$, because $1,-1\in A$.  Theorem~\ref{thm:asymmetric-main} applies.

A more genuinely non-symmetric finite example occurs in a finite cyclic group $C_m=\langle a\rangle$ with
\[
A=\{e,a\}.
\]
Then $A$ semigroup-generates $C_m$ because $a^{-1}=a^{m-1}\in A^{m-1}$, and Theorem~\ref{thm:asymmetric-main} applies.  In fact, in any finite group every non-empty set generating the group as a group also semigroup-generates it.
\end{example}

\section{Infinite sets in virtually nilpotent groups}

In this section we prove a restricted nonabelian analogue of the abelian
semilinear-set result of Biswas--Moens \cite{BiswasMoens2022}.  In abelian
groups, an unbounded linear set has the form
\[
a+\mathbb N_0 b_1+\cdots+\mathbb N_0 b_d,
\]
and finite unions of such sets are semilinear.  In a nonabelian group, there is
no canonical order-independent expression
\[
a u_1^{n_1}\cdots u_d^{n_d}
\]
unless additional commutation or normalization hypotheses are imposed.

We therefore consider sets of the form
\[
A=FM,
\]
where \(M\) is a subsemigroup containing \(e\), \(F\) is finite, and every element
of \(F\) normalizes \(M\).  The normalization condition \(F\subseteq N_G(M)\)
ensures that powers of \(FM\) have the simple form
\[
(FM)^h=F^hM.
\]

\begin{definition}[Normalizer of a subsemigroup]
Let \(M\subseteq G\) be a subsemigroup containing \(e\).  Define
\[
N_G(M):=\{g\in G:gMg^{-1}=M\}.
\]
Equivalently, \(g\in N_G(M)\) if and only if \(gM=Mg\).
It is straightforward to check that \(N_G(M)\) is a subgroup of \(G\).
\end{definition}

\begin{remark}
If \(M\) is generated as a semigroup by finitely many elements
\[
u_1,\dots,u_d\in G,
\]
then
\[
M=\langle u_1,\dots,u_d\rangle_+
:=
\{e\}\cup\{u_{i_1}\cdots u_{i_n}:n\ge1,\ i_j\in\{1,\dots,d\}\}.
\]
The results below do not require \(M\) to be finitely generated.
\end{remark}

\begin{lemma}[Powers of \(FM\)]
\label{lem:powers-FM}
Let \(M\subseteq G\) be a subsemigroup containing \(e\), and let
\(F\subseteq N_G(M)\).  Then
\[
(FM)^h=F^hM
\]
for every \(h\ge1\).
\end{lemma}

\begin{proof}
Since \(F\subseteq N_G(M)\), we have \(Mf=fM\) for every \(f\in F\).  Hence
\[
FMFM=F(MF)M=F(FM)M=F^2M.
\]
More explicitly, the inclusion \((FM)^2\subseteq F^2M\) follows from
\(MF\subseteq FM\), and the reverse inclusion follows because \(e\in M\).
The general identity \((FM)^h=F^hM\) follows by induction.
\end{proof}

\begin{proposition}[Passing from \(F\) to \(FM\)]
\label{prop:semigroup-lift}
Let \(M\subseteq G\) be a subsemigroup containing \(e\), and let
\(F\subseteq N_G(M)\) be non-empty.  If \(F\) is an asymptotic
\((r,l)\)-approximate group, then \(A:=FM\) is also an asymptotic
\((r,l)\)-approximate group.
\end{proposition}

\begin{proof}
For all sufficiently large \(h\), choose \(X_h\subseteq G\) with \(|X_h|\le l\)
and
\[
F^{rh}\subseteq X_hF^h.
\]
By Lemma~\ref{lem:powers-FM},
\[
A^{rh}
=(FM)^{rh}
=F^{rh}M
\subseteq X_hF^hM
=X_h(FM)^h
=X_hA^h.
\]
Thus \(A\) is asymptotic \((r,l)\)-approximate.
\end{proof}

\begin{proof}[Proof of theorem \ref{thm:FM-normalized}]
If \(F\) is symmetric, then \(F\) is an asymptotic approximate group by
Theorem~\ref{thm:symmetric-main}, applied inside the virtually nilpotent subgroup
\(\langle F\rangle\).  If \(F\) semigroup-generates \(\langle F\rangle\), then
\(F\) is an asymptotic approximate group by Theorem~\ref{thm:asymmetric-main}.
In both cases Proposition~\ref{prop:semigroup-lift} applies.
\end{proof}

\begin{remark}
The restriction on \(F\) in Theorem~\ref{thm:FM-normalized} is essential for
the present proof.  Proposition~\ref{prop:semigroup-lift} reduces the asymptotic
approximate-group property of \(FM\) to that of the finite set \(F\).  Without
assuming that \(F\) is symmetric or that \(F\)
semigroup-generates \(\langle F\rangle\), this would require the arbitrary finite
asymmetric case in virtually nilpotent groups.
\end{remark}

\begin{corollary}[Finite unions of cosets of a normal subgroup]
\label{cor:normal-subgroup-cosets}
Let \(G\) be virtually nilpotent, let \(N\lhd G\), and let \(F\subseteq G\) be
finite and non-empty.  Set
\[
A:=FN=\bigcup_{f\in F}fN.
\]
If \(F\) is symmetric, or if \(F\) semigroup-generates \(\langle F\rangle\), then
\(A\) is an asymptotic approximate group.
\end{corollary}

\begin{proof}
Apply Theorem~\ref{thm:FM-normalized} with \(M=N\).  Since \(N\) is normal,
\(N_G(N)=G\).
\end{proof}

\begin{corollary}[Finitely generated subsemigroups]
\label{cor:finitely-generated-subsemigroup}
Let \(G\) be virtually nilpotent.  Let \(u_1,\dots,u_d\in G\), and let
\[
M:=\langle u_1,\dots,u_d\rangle_+
\]
be the subsemigroup generated by \(u_1,\dots,u_d\).  Let
\(F\subseteq N_G(M)\) be finite and non-empty, and set
\[
A:=FM.
\]
If \(F\) is symmetric, or if \(F\) semigroup-generates \(\langle F\rangle\), then
\(A\) is an asymptotic approximate group.
\end{corollary}

\begin{proof}
This is the special case of Theorem~\ref{thm:FM-normalized} in which \(M\) is
finitely generated as a semigroup.
\end{proof}

\subsection{Stability under bounded thickening}

The preceding construction is stable under bounded enlargement.

\begin{lemma}[Bounded enlargement]
\label{lem:bounded-enlargement}
Let \(B\subseteq A\subseteq G\) be non-empty subsets.  Suppose that
\[
A\subseteq B^m
\]
for some \(m\in\mathbb N\).  If \(B\) is an asymptotic approximate group, then \(A\)
is an asymptotic approximate group.
\end{lemma}

\begin{proof}
Fix \(r\in\mathbb N\).  Since \(B\) is an asymptotic approximate group, applied
with the integer \(mr\), there are \(l\) and \(h_0\) such that
\[
B^{mrh}\subseteq X_hB^h
\]
for all \(h\ge h_0\), with \(|X_h|\le l\).  Hence
\[
A^{rh}
\subseteq (B^m)^{rh}
=B^{mrh}
\subseteq X_hB^h
\subseteq X_hA^h.
\]
Thus \(A\) is asymptotic \((r,l)\)-approximate.
\end{proof}

\begin{corollary}[Bounded thickenings of sets of the form \(FM\)]
\label{cor:bounded-thickening-FM}
Let \(G\) be virtually nilpotent, and let \(B=FM\) satisfy the hypotheses of
Theorem~\ref{thm:FM-normalized}.  If \(B\subseteq A\subseteq B^m\) for some
\(m\in\mathbb N\), then \(A\) is an asymptotic approximate group.
\end{corollary}

\begin{proof}
By Theorem~\ref{thm:FM-normalized}, \(B\) is an asymptotic approximate
group.  Now apply Lemma~\ref{lem:bounded-enlargement}.
\end{proof}

\subsection{Necessity of a normalization hypothesis}

Some normalization hypothesis is needed for the present type of nonabelian
semilinear statement.  Without such a hypothesis, a naive one-sided nonabelian
analogue of an unbounded linear set can fail to be an asymptotic approximate
group already in the integer Heisenberg group.

\begin{proposition}[Heisenberg counterexample]
\label{prop:heisenberg-normalization-fails}
Let \(H=\mathbb Z^3\) with multiplication
\[
(a,b,c)(a',b',c')=(a+a',b+b',c+c'-a'b).
\]
Put
\[
x=(1,0,0),\qquad y=(0,1,0),
\]
and
\[
A:=\{e\}\cup\{xy^n:n\ge0\}.
\]
Then, for every \(r\ge2\), every \(h\ge1\), and every finite \(X\subseteq H\),
\[
A^{rh}\nsubseteq XA^h.
\]
Consequently \(A\) is not an asymptotic \((r,l)\)-approximate group for any
\(r\ge2\) and \(l\in\mathbb N\).
\end{proposition}

\begin{proof}
We have
\[
xy^n=(1,n,0).
\]
A direct induction gives
\[
(xy^{n_1})\cdots(xy^{n_k})
=
\left(k,\sum_{i=1}^k n_i,-\sum_{i=1}^k(k-i)n_i\right).
\]
Thus every element of \(A^h\) is either \(e\), or else has the form \((k,B,C)\), where
\(1\le k\le h\), \(B\ge0\), and
\[
C\ge -(k-1)B.
\]

Fix \(r\ge2\), \(h\ge1\), and a finite set \(X\subseteq H\).  For \(n\ge0\), define
\[
g_n:=(xy^n)x^{rh-1}=(rh,n,-(rh-1)n).
\]
Since \(x=xy^0\in A\), the element \(x^{rh-1}\) is a product of \(rh-1\)
elements of \(A\).
Then \(g_n\in A^{rh}\).

Suppose, toward a contradiction, that \(g_n\in XA^h\).  Then there are
\[
t=(a,b,c)\in X
\quad\text{and}\quad
u\in A^h
\]
such that
\[
g_n=tu.
\]

First consider the case \(u=e\).  Then \(g_n=t\), so in particular \(n=b\).  For each
fixed \(t\in X\), this can occur for at most one value of \(n\).

It remains to consider \(u\ne e\).  Write
\[
u=(k,B,C),
\]
where \(1\le k\le h\), \(B\ge0\), and \(C\ge -(k-1)B\).  Comparing first and second
coordinates in
\[
(rh,n,-(rh-1)n)=(a,b,c)(k,B,C)
\]
gives
\[
a+k=rh,
\qquad
b+B=n.
\]
Comparing third coordinates gives
\[
-(rh-1)n=c+C-kb.
\]
Using \(C\ge -(k-1)B\) and \(B=n-b\), we get
\[
-(rh-1)n
=
c+C-kb
\ge
c-(k-1)(n-b)-kb
=
-(k-1)n+c-b.
\]
Therefore
\[
(rh-k)n\le b-c.
\]
But \(k\le h\), so
\[
rh-k\ge rh-h=(r-1)h\ge1.
\]
For each fixed \(t\in X\), this inequality fails for all sufficiently large \(n\).

Since \(X\) is finite, we may choose \(n\) large enough that neither the case \(u=e\)
nor the case \(u\ne e\) can occur for any \(t\in X\).  For this \(n\), we have
\[
g_n\in A^{rh}
\quad\text{but}\quad
g_n\notin XA^h.
\]
Thus
\[
A^{rh}\nsubseteq XA^h.
\]
Since \(X\) was arbitrary, \(A\) is not asymptotic \((r,l)\)-approximate for any
\(r\ge2\) and \(l\in\mathbb N\).
\end{proof}

\begin{remark}
In the Heisenberg
example, multiplying the one-sided family \(xy^n\) produces a central coordinate
whose negative slope grows with the number of factors.  A bounded number of left
translates of \(A^h\) cannot cover the larger central part in \(A^{rh}\).  This is
the obstruction that the normalization condition \(F\subseteq N_G(M)\) is designed
to avoid in Theorem~\ref{thm:FM-normalized}.
\end{remark}

\section{Functorial properties}

We note two elementary permanence properties.  They are useful when passing to
quotients by finite normal subgroups or comparing a virtually nilpotent group with a
finite-kernel quotient.

The first statement is the asymptotic analogue of Nathanson's homomorphism
lemma for approximate groups; compare \cite[Lemma~5]{Nathanson2018}.

\begin{proposition}[Homomorphic images]\label{prop:homomorphic-images}
Let \(\pi:G\to Q\) be a group homomorphism, and let \(A\subseteq G\) be non-empty.
If \(A\) is asymptotic \((r,l)\)-approximate in \(G\), then \(\pi(A)\) is asymptotic
\((r,l)\)-approximate in \(Q\).
\end{proposition}

\begin{proof}
For all sufficiently large \(h\), choose \(X_h\subseteq G\) with \(|X_h|\le l\) and
\[
A^{rh}\subseteq X_hA^h.
\]
Applying \(\pi\), we get
\[
\pi(A)^{rh}
=
\pi(A^{rh})
\subseteq
\pi(X_h)\pi(A^h)
=
\pi(X_h)\pi(A)^h.
\]
Since \(|\pi(X_h)|\le |X_h|\le l\), the result follows.
\end{proof}

\begin{proposition}[Finite-kernel lifting]\label{prop:finite-kernel-lifting}
Let \(\pi:G\to Q\) be a group homomorphism with finite kernel \(K\).  Let
\(A\subseteq G\) be non-empty.  If \(\pi(A)\) is asymptotic \((r,l)\)-approximate in
\(Q\), then \(A\) is asymptotic \((r,|K|l)\)-approximate in \(G\).
\end{proposition}

\begin{proof}
For all sufficiently large \(h\), choose \(Y_h\subseteq Q\) with \(|Y_h|\le l\) and
\[
\pi(A)^{rh}\subseteq Y_h\pi(A)^h.
\]
We first replace \(Y_h\) by
\[
Y_h':=Y_h\cap \pi(G).
\]
This does not destroy the covering property.  Indeed, if
\(z\in \pi(A)^{rh}\), then \(z\in \pi(G)\).  If \(z=yb\) with
\(y\in Y_h\) and \(b\in \pi(A)^h\), then \(b\in \pi(G)\), and hence
\[
y=zb^{-1}\in \pi(G).
\]
Thus
\[
\pi(A)^{rh}\subseteq Y_h'\pi(A)^h
\quad\text{and}\quad
|Y_h'|\le |Y_h|\le l.
\]

For each \(y\in Y_h'\), choose one lift \(\widetilde y\in G\) with
\[
\pi(\widetilde y)=y,
\]
and set
\[
X_h:=\{\widetilde y k:y\in Y_h',\ k\in K\}.
\]
Then
\[
|X_h|\le |K|\,|Y_h'|\le |K|l.
\]

Let \(g\in A^{rh}\).  Then
\[
\pi(g)\in \pi(A)^{rh}\subseteq Y_h'\pi(A)^h.
\]
Thus there exist \(y\in Y_h'\) and \(a\in A^h\) such that
\[
\pi(g)=y\pi(a).
\]
Using the chosen lift \(\widetilde y\), this gives
\[
\pi(\widetilde y^{-1}ga^{-1})=e.
\]
Hence
\[
\widetilde y^{-1}ga^{-1}\in K.
\]
Therefore
\[
g\in \widetilde y K A^h\subseteq X_hA^h.
\]
Since \(g\in A^{rh}\) was arbitrary,
\[
A^{rh}\subseteq X_hA^h.
\]
Thus \(A\) is asymptotic \((r,|K|l)\)-approximate.
\end{proof}

\section{Further comments and possible direction}\label{sec:furthercomments}

The symmetric argument exploits the identity \(A^h=B_A(h)\), while the
semigroup-generating asymmetric argument relies on the weaker fact that
\(A^h\) still contains a symmetric word ball whose radius grows linearly with
\(h\).  For a general finite set \(A\subseteq G\) with \(e\in A\) that generates
\(\Gamma=\langle A\rangle\) as a group but not as a semigroup, neither feature
need be present.  In that situation the sets \(A^h\) may resemble a directed
cone, or even a lower-dimensional region, inside the asymptotic cone. One possible way could be a directed limit-shape theorem for the
one-sided product sets \(A^h\), parallel in spirit to the asymptotic-shape
theory for symmetric word balls developed by Pansu \cite{Pansu1983} and by Breuillard \cite{Breuillard2014}. Given a finitely
generated virtually nilpotent group \(\Gamma\) and a finite set
\(A\subseteq\Gamma\) with \(e\in A\), one would like to know whether the rescaled
sets \(A^h\) converge in the asymptotic cone to a compact homogeneous set
\(K_A\), perhaps lower-dimensional and with \(e\) lying on its boundary.  A strengthening would be an interior-point criterion: whenever
\(g_h\in\Gamma\) converges after rescaling to a point in a suitable interior
region of \(K_A\), the elements \(g_h\) should eventually belong to \(A^h\).  If
available, such a statement could be combined with a compactness argument to
cover the limit shape for \(A^{rh}\) by finitely many left translates of a
slightly eroded copy of the limit shape for \(A^h\), and then transfer that
covering back to the discrete group.

\end{document}